\documentclass[12pt]{amsart}

\usepackage{geometry}
\usepackage[utf8]{inputenc}
\usepackage{bm}
\usepackage{fontenc}
\usepackage{wasysym}
\usepackage{amsfonts}
\usepackage{amssymb}
\usepackage{amsmath}
\usepackage{amsthm}
\usepackage{mathtools}
\usepackage{enumerate}
\usepackage{mathrsfs}
\usepackage{tikz}
\usetikzlibrary{calc}
\usepackage{marginnote}
\usepackage{xcolor,enumitem}
\usepackage{soul}
\usepackage{hyperref}
\usepackage[all,cmtip]{xy}

\usepackage{tikz-cd}
\usepackage{xcolor}
\usepackage{hyperref}

\usetikzlibrary{matrix}

\newtheorem*{maintheorem}{Main Theorem}

\newtheorem{theorem}{Theorem}[section]
\newtheorem*{theorem*}{Theorem}
\newtheorem{proposition}[theorem]{Proposition}

\newtheorem*{proposition*}{Proposition}
\newtheorem{lemma}[theorem]{Lemma}
\newtheorem*{lemma*}{Lemma}
\newtheorem{corollary}[theorem]{Corollary}
\newtheorem*{corollary*}{Corollar}
\newtheorem{fact}[theorem]{Fact}
\newtheorem*{fact*}{Fact}

\theoremstyle{definition}
\newtheorem{definition}[theorem]{Definition}
\newtheorem*{definition*}{Definition}
\newtheorem{claim}[theorem]{Claim}
\newtheorem*{claim*}{Claim}

\newtheorem*{conjecture*}{Conjecture}

\theoremstyle{remark}

\newtheorem*{example*}{Example}
\newtheorem{remark}[theorem]{Remark}
\newtheorem*{remark*}{Remark}

\newtheorem*{note*}{Note}
\newtheorem*{question*}{Question}

\newtheorem*{acknowledgements*}{Acknowledgements}

\newcommand{\Z}{\ensuremath{\mathbb{Z}}}

\newcommand{\A}{\ensuremath{\mathbf{A}}}
\newcommand{\B}{\ensuremath{\mathbf{B}}}

\newcommand{\G}{\ensuremath{\mathbf{G}}}
\newcommand{\Q}{\ensuremath{\mathbf{Q}}}
\newcommand{\Zig}{\ensuremath{\mathbf{Z}}}

\newcommand{\bfk}{\ensuremath{\mathfrak{b}}}
\newcommand{\dfk}{\ensuremath{\mathfrak{d}}}
\newcommand{\cfk}{\ensuremath{\mathfrak{c}}}

\newcommand{\F}{\ensuremath{\mathcal{F}}}
\newcommand{\J}{\ensuremath{\mathcal{J}}}
\newcommand{\cA}{\ensuremath{\mathcal{A}}}
\newcommand{\cB}{\ensuremath{\mathcal{B}}}

\newcommand{\lang} {\ensuremath{\langle} }
\newcommand{\rang}{\ensuremath{ \rangle} }
\newcommand{\lB}{\ensuremath{\lbrace} }
\newcommand{\rB}{\ensuremath{ \rbrace} }
\newcommand{\et}{ \ensuremath{\wedge} }

\newcommand{\sbeq}{\ensuremath{ \subseteq }}

\newcommand{\defeq}{\vcentcolon=}

\newcommand{ \xx }{ \Bar{x}   }
\newcommand{ \yy }{ \Bar{y}   }
\newcommand{ \zz }{ \Bar{z}   }
\newcommand{ \uu }{ \Bar{u}   }
\newcommand{ \vv }{ \Bar{v}   }
\newcommand{ \ww }{ \Bar{w}   }

\newcommand{ \lam }{ \ensuremath{\lambda}   }
\newcommand{ \Lam }{ \ensuremath{\Lambda}   }
\newcommand{ \cof }{ \ensuremath{\textrm{cof}  }   }

\renewcommand{\Diamond}{\lozenge}
\renewcommand{\overline}[1]{\Bar{#1}}

\begin{document}

\title{Nonvanishing derived limits without scales}

\author{Matteo Casarosa}
\address{
Institut de Math\'ematiques de Jussieu - Paris Rive Gauche (IMJ-PRG)\\
Universit\'e Paris Cit\'e\\
B\^atiment Sophie Germain\\
8 Place Aur\'elie Nemours \\ 75013 Paris, France, and Dipartimento di Matematica, Universit\`{a} di Bologna, Piazza di
Porta S. Donato, 5, 40126 Bologna,\ Italy}
\email{matteo.casarosa@imj-prg.fr}
\urladdr{https://webusers.imj-prg.fr/~matteo.casarosa/}

\email{matteo.casarosa@unibo.it}
\urladdr{https://www.unibo.it/sitoweb/matteo.casarosa/en}



\begin{abstract}

The derived functors $\lim^n$ of the inverse limit are widely studied for their topological applications, among which are some repercussions on the additivity of strong homology. Set theory has proven useful in dealing with these functors, for instance in the case of the inverse system $\mathbf{A}$ of abelian groups indexed by ${}^\omega \omega$. So far, consistency results for nonvanishing derived limits of $\A$ have always assumed the existence of a scale (i.e. a linear cofinal subset of $({}^\omega \omega, \leq^\ast )$, or equivalently that $\mathfrak{b} = \mathfrak{d} $). Here we do away with that assumption and prove that nonvanishing derived limits, and hence the non-additivity of strong homology, are consistent with any value of $\aleph_1 \leq \mathfrak{b} \leq \mathfrak{d} < \aleph_\omega$, thus giving a partial answer to a question of Bannister.

\end{abstract}

\maketitle

\section{Introduction}

This is an article on applications of set theory to homological algebra, and more precisely to the derived functors of the inverse limit. These are usually denoted by $\lim^n$ for $n > 0 $ and they give an estimate of how far the inverse limit is from being an exact functor.  

In recent years, set theory has proven useful to deal with some problems related to derived limits. These problems have repercussions in algebraic topology. For instance, the non-additivity of $\textit{strong homology}$ is implied by the nonvanishing of any of the derived limits $\lim^n$ of a certain pro-group $\A$ indexed on ${}^\omega \omega$ (\cite{mardevsic1988strong}), while the vanishing of $\lim^n$ for all $n$ for a certain class of pro-groups that generalizes $\A$ guarantees the additivity of strong homology on the class of locally compact separable metric spaces (\cite{bannister2020additivity}, \cite{bannister2023additivity}). As it happens oftentimes with set theory, many questions on this subject turn out to be independent of the $\mathrm{ZFC}$ axioms, and many theorems take the form of consistency results. On the nonvanishing side, the historical progression is as follows:

\begin{enumerate}

    \item $\mathfrak{d} = \aleph_1$ implies that $\lim^1 \A \neq 0$ (\cite{dow1989strong}).

    \item $\lim^1 \A^\mathcal{F} \neq 0$ is a $\mathrm{ZFC}$ fact, whenever $\mathcal{F}$ is an $\omega_1$-chain in $({}^\omega \omega , \leq^\ast)$, and $\A^\mathcal{F}$ is the restriction of $\A$ to the indexing subset $\mathcal{F}$ (\cite{bekkali2006topics}). 

    \item  It is consistent that $\lim^1 \A \neq 0$, $\mathfrak{c} = \aleph_2$ and $\mathrm{MA}_{\aleph_1} $ holds (and hence $\mathfrak{b} = \mathfrak{d} = \aleph_2$    (\cite{todorcevic1998first})).

\end{enumerate}

Then there is a small series of nonvanishing results for higher derived limits, namely, the functors $\lim^n$ for $n > 1$.

\begin{enumerate}

\setcounter{enumi}{3}

    \item It is consistent that $\lim^n \Zig_n \neq 0$ for all $n$ were $\Zig_n$ is a certain pro-group indexed on the ordinal $\omega_n$ (\cite{ziegler1999higher}). In particular, the case $n =1$ is a $\mathrm{ZFC}$ theorem that can be proved for instance using the $\rho_1$ function (\cite{todorcevic2007walks}) and one can then use the guessing principles $\Diamond (S^i_{i+1})$ for $i<n$ to run an induction through higher $n$.   

    \item It is consistent that $\lim^2 \A \neq 0$. In particular, it is implied by $\mathfrak{b} = \mathfrak{d} = \aleph_2 + \Diamond(S^1_2) $(\cite{berg}). Here once again one starts with a $\mathrm{ZFC}$ base case and then moves to the higher case using the guessing principle. Note that the condition $\mathfrak{b} = \mathfrak{d} = \kappa$ is equivalent to the existence of a cofinal $\kappa$-chain in $({}^\omega \omega , \leq^\ast)$, that is, a $\kappa$-scale. This allows us to work in a linear situation pretty much like in the case of $\Zig_n$. Note also that $ \Diamond(S^2_1) $ implies $\mathfrak{c} \leq \aleph_2$, and this is the obstruction to using the classical diamonds to get beyond $n = 2$.

    \item For all $n >0$ it is consistent that $\lim^n \A \neq 0$. In particular, this is implied by $\mathfrak{b} = \mathfrak{d} = \aleph_n + \bigwedge_{i < n} w\Diamond (S^i_{i+1}) $ (\cite{velickovic2021non}). The weak diamonds $w \Diamond(S)$ are some guessing principles that were originally introduced by Devlin and Shelah in \cite{devlin1978weak}. These are sufficient to run an argument similar to that for $\Zig_n$ and are compatible with a longer scale. 
    
\end{enumerate}

Since all nonvanishing results for this kind of systems so far assume or entail a linear indexing set (or cofinal subset thereof) it is natural to ask what the behavior of the derived limits of $\A$ is in case $\mathfrak{b} < \mathfrak{d}$. Indeed, this problem is mentioned among the open questions discussed in \cite{bannister2023additivity} and \cite{bergfalk2023simultaneously}. In this paper, we are going to discuss precisely that question, and more precisely we are going to prove the following.

\begin{maintheorem}

\label{maintheorem}

For all $ 1 \leq k \leq n < \omega$, it is consistent that $\bfk= \aleph_k$, $\dfk= \aleph_n$ and $\lim^n \A \neq 0$.
\end{maintheorem}

This in turn implies (see \cite{mardevsic1988strong}) the following.

\begin{corollary}

For all $1 \leq k \leq n < \omega$, it is consistent that $\bfk= \aleph_k$, $\dfk= \aleph_n$, and strong homology is not additive on separable locally compact metric spaces.
    
\end{corollary}

This is in some way complementary to the work of Bannister in \cite{bannister2023additivity} where, by adding $\aleph_\omega $-many Cohen reals over a model of $\mathrm{GCH}$, additivity is forced for that same class of spaces.

The structure of the present paper is the following. First, we are going to prove by induction that the existence of an unbounded chain of a certain length in $({}^\omega \omega, \leq^\ast)$, together with some guessing principles, implies $\lim^n \A \neq 0$. This we prove by induction starting from a base case where we show the presence of coherent families that have some nontriviality concentrated in a particular part of them. This stronger nontriviality condition is then shown to be preserved in the induction. Then we present a forcing that shows the consistency of the assumptions used in the inductive argument with all the relevant values of $\mathfrak{b}$ and $\mathfrak{d}$. Finally, in the appendix we present a proof of Roos' characterization of derived limits, which is used for this kind of set-theoretic arguments.

\subsection*{Acknowledgements}

I thank Jeffrey Bergfalk, Chris Lambie-Hanson and Alessandro Vignati for helpful conversations. Part of these happened during a visit to the University of Barcelona funded by the FSMP, and a visit to the Institute of Mathematics of the Czech Academy of Sciences, partly funded by the Starting Grant 101077154 ``Definable Algebraic Topology" from the European Research Council, and partly by the GAČR project 23-04683S. Finally, I thank Filippo Callegaro for supervising my master's thesis from which the appendix is taken.

\section{Preliminaries and notation}

Let $\mathbf{G} = ( G_{\lambda}, p_{\lambda}^{\mu} , \Lambda )  $ be an inverse system of abelian groups, with groups $G_{\lambda}$  indexed over the directed quasi-order   $(\Lambda, \leq) $ and bonding morphisms $p^{\mu}_{\lambda}: G_{\mu} \to G_{\lambda} $, with $\lambda \leq \mu \in \Lambda$, between them. Let moreover $\Lambda^{(n)}$ be the set of $\leq$-ordered $n+1$-tuples in $\Lambda$. If $\overline{\lambda} = (\lambda_0, ... , \lambda_n ) \in \Lambda^{(n)} \, $we write $\overline{\lambda^i} = (\lambda_0, ..., \lambda_{i-1}, \widehat{\lambda_i}, \lambda_{i+1}, ..., \lambda_n) $, meaning that we have removed the element $\lambda_i$.

The derived limit $\lim^n$, for $n > 0$, is the $n$-th derived functor of the inverse limit $\lim$. We can also write $\lim^0$ for $\lim$. We refer the reader interested in the definition of derived functors through resolutions or the total derived functor to \cite{weibel1995introduction} and \cite{schneiders1999quasi} respectively. Throughout the proof of the main theorem, we will use the following characterization due to Roos (\cite{roos1961foncteurs}).

\begin{definition}[Roos complex]
    
\label{defrooscomplex}

We define the cohomological Roos complex as 

\[ K^{\bullet} (\mathbf{G}) \, = \, ( \, 0 \to    K^0 (\mathbf{G}) \xrightarrow[]{\delta^1}    K^1 (\mathbf{G})  \xrightarrow[]{\delta^2}    K^2 (\mathbf{G}) \xrightarrow[]{\delta^3} ... \, ),  \]
where

\[  K^n (\G) = \prod_{\overline{\lambda} \in \Lambda^{(n)} } X_{\lambda_0 } = \prod_{\lambda_0 \leq ... \leq \lambda_n } G_{\lambda_0}. \]
In other words, $K^n $ consists of all functions $\overline{x} : \Lambda^{(n)} \to \bigcup_{\lambda} G_{\lambda}  $ (we think of the $G_{\lambda}$ as disjoint) such that $\overline{x} (\overline{\lambda})  \in  G_{\lambda_0} $. We will also write $x_{\overline{\lambda}}$ for $\overline{x} (\overline{\lambda}) $ and $\G^{(n)}$ for $K^n(\G)$ and $\G^{(n)} \restriction C$ for $K^n(\G \restriction C)$.  The coboundary maps are 

\[   \delta^n (\overline{x})_{\overline{\lambda}} = p_{\lambda_0}^{\lambda_1} (x_{\overline{\lambda^0}}  )  + \sum_{1 \leq i \leq n} (-1)^i x_{\overline{\lambda^i}}.                                   \]
Note that we stipulate that $K^{-1} (\G) = 0$ and $\delta^0 = 0$.

\end{definition}

We also define a couple of terms related to this construction.

\begin{definition}[Coherent and trivial elements]

We call an element  $n$-\textit{coherent}  if it is in the kernel of $\delta^{n+1} $, and we call it $n$-\textit{trivial} if it is in the image of $\delta^n$. If $\delta^n (\overline{y}) = \xx$ we say that $\overline{y}$ \textit{trivializes} $\xx$. 

\end{definition}

The derived limits can be computed as the cohomology of the Roos complex.

\begin{theorem}

Let $\G$ be an inverse system of groups, $K^\bullet $ be the associated Roos complex, and $n \geq 0$. Then

\begin{center}

$ \lim^n \G \cong H^n (K^\bullet)   $.
\end{center}
    
\end{theorem}

This is part of a theorem for which the literature on applications of set theory to this topic usually references the rather long proof of \cite{mardesic2000strong} in the case of modules. In the appendix, we offer a shorter categorical proof that may be more amenable to logicians. There we expand upon the argument of \cite{neeman2014triangulated} by clarifying some constructions and proving an intermediate claim that may be obscure to some readers.

We now record two basic results we will need in the next section.

\begin{definition}

Let 

\[   \textbf{G} = \lB G_{\lambda} ; p^{\mu}_{\lambda} : \lambda \leq \mu \in \Lambda    \rB  .       \]  
be an inverse system of abelian groups. We say that the system is surjective if each $p^{\mu}_{\lam} $ is surjective.

\end{definition}

\begin{theorem}[Surjective Goblot's Theorem \cite{velickovic2021non}]

\label{flasquegoblot}

Let $n \geq 0$. If $\G$ is a surjective inverse system of abelian groups with $\cof( \G) = \aleph_n $, then $\lim^k \G = 0 $, for all $k \geq n+1$.

\end{theorem}

If $M$ is a directed subset of the indexing set $\Lam$ we will write  $\G \restriction M$ for the system $\lB G_{\lambda} ; p^{\mu}_{\lambda} : \lambda \leq \mu \in M    \rB  $.

\begin{theorem}[Mitchell, 1973]

\label{mitchell}

Let $\G$ be an inverse system indexed by $\Lam$, let $M$ be another directed set and let $\phi: M \to \Lam$ be an order preserving cofinal map. Let moreover $\phi^\ast(\G)$ be the system indexed on $M$ defined by pre-composing the functor  $\phi$ between the two orders thought of as categories to the functor $\G$.  Then, for all $n \geq 0$, \\

\begin{center}    $     \lim^n \G \cong \lim^n \phi^* (\G).   $  \end{center}

In particular, the case where $\phi$ is an inclusion gives us that for a cofinal subset $C \sbeq \Lam$, we have $\lim^n \G \cong \lim^n \G \restriction C$.

\end{theorem}

To more easily apply the last two theorems we will introduce the setting of quotient systems described in the following section and will apply it to the systems in which we are interested.

\section{The inverse system $\mathbf{A}$ and its relatives}

The main object of our interest is the inverse system $\A$. It is indexed by the set ${}^\omega \omega $ of all functions $f: \omega \to \omega $ and such that the object $A_f$ corresponding to $f$ is defined as

\[  A_f = \bigoplus_{i \in \omega} \Z^{f(i)}.               \]
The order relation on this set is everywhere domination (i.e., $f   \leq g$ if and only if $f(i) \leq g(i) $ for all $i \in \omega$). The bonding morphisms are the obvious projections. We write $I_f = \lB (i, j) \mid j \leq f(i) \rB$ for the region on which the elements of the groups seen as functions to $\Z$ are supported. Analogously, we define the system $\B$ where we replace infinite sums with infinite products:

\[ B_f = \prod_{i \in \omega} \Z^{f(i)}.  \] 
This yields a short exact sequence of inverse systems

\[  0 \to \A \to \B \to \B / \A \to 0,    \]
which induces a long exact sequence

\begin{center}

$ 0 \to  ... \to \lim^n \A \to \lim^n \B \to \lim^n \B / \A \to \lim^{n+1} \A \to ...  \,\,\,\,\,\, .  $

\end{center}

Now, \cite[Lemma 4]{mardevsic1988strong} says that $\lim^n \B = 0$ for $n \geq 1$, so that $\lim^{n+1} \A \cong \lim^n \B / \A$ and $\lim^1 \A \cong \frac{\lim^0 \B/ \A}{\lim^0 \B}$. The proof straightforwardly generalizes to restrictions to directed subsets and groups different than $\Z$. In other words, for any directed $\mathcal{F} \sbeq {}^\omega \omega$, and any nontrivial abelian group $G$, we can define $\A^{\mathcal{F}}_G$ and $\B^{\mathcal{F}}_G$ whose objects are $A_f = \bigoplus_{i \in \omega} G^{f(i)} $ and $B_f = \prod_{i \in \omega} G^{f(i)}$ respectively, for each $f \in \F$. Then the reasoning above still yields isomorphisms $\lim^1 \A^\F_G \cong \frac{\lim^0 \B^\F_G / \A^\F_G}{\lim^0 \B^\F_G} $ and $\lim^{n+1} \A^\F_G \cong \lim^n \B^\F_G / \A^\F_G$.

For the case $n > 1$, we further observe that the inclusion map $\iota : (\F , \leq ) \hookrightarrow (\F , \leq^*) $ is order-preserving and cofinal, hence by Theorem \ref{mitchell} we can further substitute for $\lim^n \B^\F_G / \A^\F_G $ the limit $\lim^n (\B^\F_G / \A^\F_G)_{\leq^\ast} $ of a system with the same quotient groups as objects and with bonding morphisms analogous to the previous ones for each instance of $\leq^*$. One can justify the same conclusions by the approach of \cite{velickovic2021non}. This last system is surjective and has the property of being indexed on the very order whose cofinality is by definition $\dfk$. This will be relevant for the use we are going to make of Theorem \ref{flasquegoblot}.

Let $=^\ast $ stand for equality except possibly on finitely many points. The isomorphism for $\lim^1 $ yields the following characterization, where one can indifferently use $\omega \times \omega$ or $\bigcup_{f \in \F} I_f$ as the domain of $\psi$.

\begin{proposition}

Let $G$ be a nontrivial group, and let $\F\sbeq {}^\omega \omega$ be directed. Then $\lim^1 A^\F_G \neq 0 $ if and only if there exists a family $\lang \phi_f: I_f \to G \mid f \in \F     \rang$ of functions such that:

\begin{enumerate}
    \item For every two $f,g \in \F$, $\phi_f \restriction I_f \cap I_g =^\ast \phi_g \restriction I_f \cap I_g $. 
    
    \item There exists no function $\psi: \omega \times \omega \to G$ such that $\psi \restriction I_f =^\ast \phi_f$ for all $f \in \F$.

    \end{enumerate}
\end{proposition}

We call a family that satisfies the first condition \textit{coherent}. If it satisfies the second condition we say moreover that it is \textit{nontrivial}.

\begin{remark}

\label{remarkcardinality}

Note that $\lim^1 \A^\F_G \neq  0$ implies $\lim^1 \A^\F_{G'} \neq 0$ whenever $\vert G \vert \leq \vert G' \vert$, so that the vanishing of $\lim^1 \A^\F_G$ is equivalent for groups of the same cardinality, while the group structure is relevant for higher limits. This is because the equation $a-b= 0$ reduces to the logical notion of equality $a=b$, so that a witness for  $\lim^1 \A^\F_G \neq 0$ can be seen modulo injection as a witness for $\lim^1 \A^\F_{G'} \neq 0$, for any $G'$ of greater or equal size, in contrast, there is an inescapable algebraic aspect to longer alternating sums.

\end{remark}

The characterization for higher derived limits can be expressed in terms of $n$-coherent and $n$-nontrivial families.

\begin{proposition}

\label{ncornontriv}

For $n\geq 1$, $\lim^{n+1} \A^\F_G \neq 0$ if and only if there exists a family $\xx \in ((\B^\F_G / \A^\F_G)_{\leq^\ast})^{(n)}$ such that $\delta^{n+1} (\xx) = 0$ and such that for no $\yy \in  ((\B^\F_G / \A^\F_G)_{\leq^\ast})^{(n-1)}$ we have $\delta^n(\yy) = \xx$. 
    
\end{proposition}

\section{Unbounded nontriviality}

Before we start the proof by induction of the main result of this section, we need a lemma.

\begin{definition}

Let $(P, \leq)$ be a quasi-order. We say that $C \sbeq P$ is a $\kappa$-\textit{chain} if there is a bijection $\sigma: \kappa \to C$ such that $\alpha \leq \beta$ implies $\sigma(\alpha) \leq \sigma(\beta)$.
    
\end{definition}

\begin{remark}
\label{remarkstrictchain}
If $\kappa$ is regular and $C$ is unbounded in $P$, then we can find $C' \sbeq C$ ordered by the strict order $<$ such that $a < b \leftrightarrow (a \leq b \et \neg (a \geq b)) $.
    
\end{remark}

\begin{lemma}

\label{increasingunion}

Let $n \geq 0$ and $(P, \leq) $ be a quasi-order such that $\cof(P) = \aleph_{n+1}$ and there exists an $\omega_{n+1}$-chain $C$ that is unbounded in $P$. Let moreover $P$ be equipped with a function $d: P^2 \to P$ such that:

\begin{enumerate}

\item $d(x, y) \geq x, y$ (witnessing the directedness of $P$);

\item $d(x', y') \leq d(x, y) $ whenever $x' \leq x$ and $y' \leq y$. 

\end{enumerate}

Then we can find a continuous increasing union 

\[ P = \bigcup_{\alpha < \omega_{n+1}} P_\alpha,   \]
and a $C' \sbeq C$ cofinal in $C$ such that for each $\alpha < \omega_{n+1} $, $P_\alpha$ is $d$-closed, downward closed, with  $\cof(P_\alpha) \leq \aleph_n$, $\mathrm{otp}(P_\alpha \cap C') = \alpha$ and $ P_\alpha \cap C'$ unbounded in $P_\alpha$ for every limit $\alpha$.
    
\end{lemma}

\begin{proof}

Let $c: \omega_{n+1} \to P$ be an enumeration of a cofinal subset of $P$, and assume without loss of generality, thanks to Remark \ref{remarkstrictchain}, that $C$ is a chain of order-type $\omega_{n+1}$ for the strict order relation.

We construct by mutual transfinite recursion the increasing union and a function $b: \omega_{n+1} \to C$ such that $b(\alpha) $ is the least element of $C \setminus P_\alpha$, where $P_\alpha$ is the downward closure of the $d$-closure of $c[\alpha] \cup b[\alpha]$ (note that it stays $d$-closed by condition 2). This is increasing and continuous by construction. The $d$-closure of $c[\alpha] \cup b[\alpha]$ is cofinal in $P_\alpha$ and of cardinality $\leq \aleph_n$. For each element in it, we can find one in the chain that is not dominated by it; their supremum (which exists because the chain has cofinality $> \aleph_n$) will not be in the $d$-closure $P_\alpha$. This shows that $b$ is well-defined, strictly increasing, and hence cofinal. Finally, we put $C' = b[\omega_{n+1}]$. The conditions on the order-type and unboundedness are immediate by construction.
\end{proof}

\begin{remark}
    
\label{remarkunbounded}

Note that, for every limit $\lambda$, $\cof(P_\lambda) \geq \cof(\lambda)$. This follows from the fact that for every limit $\lambda < \omega_{n+1} $,  $C' \cap P_\lambda$ is unbounded in $P_\lambda$. Otherwise, for each element in a cofinal set of lesser cardinality we pick one of the chain that is not dominated by it, and taking the supremum we get a contradiction, as in the previous proof.

\end{remark}

\begin{remark}

In case $P= {}^\omega \omega $, one natural choice for the function $d$ is the lowest upper bound $\vee : ({}^\omega \omega)^2 \to {}^\omega \omega   $ defined as $(f \vee g)(i) = \max \lB f(i) , g(i) \rB  $. The same is true for every subset $P \sbeq {}^\omega \omega$ that is $\vee$-closed. 
    
\end{remark}

The proof in \cite{velickovic2021non} served as a basis for the following argument. The central idea allowing the present advances is the construction of some nontriviality for the base case having the additional property of being ``localized", meaning that the restriction of a larger coherent family to a certain chain is nontrivial, and the observation that this additional property is preserved in the induction. 

\subsection{Base case}

To get a result that holds for all nontrivial groups $G$, we need an extra cardinal arithmetic assumption for the base case.

\begin{lemma}

\label{basecase}

Assume that $2^{\aleph_0} < 2^{\aleph_1}$. Let $P \sbeq ({}^\omega \omega , \leq^\ast ) $ be $\vee$-closed and such that $\cof(P) = \aleph_1$ and there exists an $\omega_1$-chain $C$ which is unbounded in $P$, and let $G \neq 0$ be an abelian group. Then there is a coherent 1-family $\Phi = \lang \phi_f: I_f \to G \mid f \in P \rang $ such that $\Phi \restriction C$ is (coherent) nontrivial.

\end{lemma}

\begin{proof}

By the reasoning that shows Remark \ref{remarkcardinality}, we can restrict to the case $G = \Z/2\Z$.

Apply Lemma \ref{increasingunion} with respect to $d = \vee$ to get a continuous increasing union $\bigcup_{\alpha < \omega_1} P_\alpha = P $ and $C'$ satisfying the lemma. Recall in particular that the $\alpha$-th element of $C'$ (henceforth $g_\alpha$) is not in the downward-closed set $P_\alpha$, and $\cof(P_\alpha) \leq \omega$. So let $\lB \hat{f}_{\alpha,n} \mid n < \omega \rB $ be a cofinal subset of $P_\alpha$. Since $P_\alpha$ is $\vee $-directed we can let $f_{\alpha,n} = \bigvee_{i \leq n}  \hat{f}_{\alpha, i}$ so that $\lang f_{\alpha, n} \mid n < \omega \rang $ is cofinal for $\leq^\ast$ and a chain with respect to $\leq$. 

Since $g_\alpha \not \leq^\ast f_{\alpha,n} $, the set $I_{g_\alpha} \setminus I_{f_{\alpha,n}}$ is infinite (in particular non-empty) for all $n$. Pick $x_{\alpha, n} \in I_{g_\alpha} \setminus I_{f_{\alpha, n}} $. Then $X_\alpha = \lB x_{\alpha, n} \rB_{n < \omega} \sbeq I_{g_\alpha}$ is infinite and almost disjoint from all of the $I_{f_{\alpha, n}} $ as $X_\alpha \cap I_{f_{\alpha, n}} \sbeq \lB x_{\alpha, 0} , ... , x_{\alpha, n-1} \rB$.

We are now going to define a tree of coherent families ordered by end-extension, such that among the unions of the branches one yields a nontrivial family by counting arguments.

For all $s \in 2^{< \omega_1} $ we iteratively construct a coherent family $\Phi^s$ on $P_\alpha$, where $\textrm{dom}(s) =\alpha$ as follows. If $\alpha \in \lim(\omega_1)$, then $\Phi^s = \bigcup_{\alpha < \textrm{dom}(s)} \Phi^{s \restriction \alpha}$. Suppose now that $\alpha = \beta +1$, we have already defined $\Phi^s$ with $\textrm{dom}(s) = \beta$ and we want to define $\Phi^{s, 0}$ and $\Phi^{s,1}$. So we find a trivialization $\psi^s$ of $\Phi^s$ that exists since $P_\beta$ is of countable cofinality, so we can consider its cofinal chain $\lang f_{\beta, n} \mid n < \omega \rang $ defined above; it is then easy to see that we can inductively find finite sets $F_n $ such that for all $m \leq n $ we have that $\phi^s_{f_{\beta, m}} \restriction (I_{f_{\beta, m}} \setminus F_m )$ and $ \phi^s_{f_{\beta, n}} \restriction (I_{f_{\beta, n}} \setminus F_n )$ coincide on their common domain. We can then define 

\[   \psi^s (x) = \begin{cases}
   \phi^s_{f_{\beta, n}} \,\,\,\,  \textrm{if} \, x \in (I_{f_{\beta,n}} \setminus F_n ) \,\, \textrm{for some} \, n   \\

   0 \,\,\,\,  \textrm{otherwise}.

\end{cases}             \]
Then we let $\psi^{s,0} = \psi^s$ and
\[   \begin{cases}    \psi^{s,1} (x) = \psi^s (x) +1  \,\,\,\, \textrm{if} \, x \in X_\beta 

\\

\psi^s (x)   \,\,\,\,  \textrm{otherwise}. 

\end{cases}        \]
Finally, we let $\Phi^{s,0} = \delta^1(\psi^{s,0}) $ and $\Phi^{s,1} = \delta^1 (\psi^{s,1})$. 

Now for all $h: \omega_1 \to  2$ we put $\Phi^h = \bigcup_{\alpha < \omega_1} \Phi^{h \restriction \alpha}$. Note that if $h \neq h'$, then $\phi^h_{g_\gamma} \neq^\ast \phi^{h'}_{g_\gamma}$ where $\gamma$ is such that $h(\gamma) \neq h'(\gamma)$. But then the same trivialization cannot trivialize both $\Phi^h \restriction C'$ and $\Phi^{h'} \restriction C'$. Since there are only $2^{\aleph_0}$ many putative trivializations and we are assuming that $2^{\aleph_0} < 2^{\aleph_1}$, there must be $\overline{h}: \omega_1 \to 2$ such that $\Phi^{\overline{h}} \restriction C'$ is nontrivial. Hence, by Theorem \ref{mitchell}, $\Phi^{\overline{h}} \restriction C$ is also coherent nontrivial. 
\end{proof}

\begin{remark}

\label{remarkzeta}

In case $G$ is infinite, the cardinal arithmetic assumption $2^{\aleph_0} < 2^{\aleph_1}$ is not needed, as one can simply apply the reasoning of \cite[Proposition 3.2]{velickovic2021non} with the sets $X_\alpha$ above instead of the sets $b_\xi$ of the cited paper, after fixing a bijection between $\omega \times \omega $ and $\omega$, and an injection of the latter into $G$.
    
\end{remark}

\subsection{Induction step}

\begin{definition}
We say that $R(n,G)$ holds for some group $G$ if for every $\vee$-closed $P_n \sbeq ({}^\omega \omega , \leq^\ast) $ with $\cof(P_n) = \aleph_n$ that contains an $\omega_n$-chain $C_n$ that is unbounded in it, there exists an $n$-coherent family on $P_n$ for $G$ such that its restriction to $C_n$ is nontrivial.  

\end{definition}

Whenever $S$ is a stationary set, $w \Diamond (S)$ denotes a guessing principle called \textit{weak diamond} of $S$, introduced in \cite{devlin1978weak}. Recall that $S^n_{n+1} = \lB \alpha < \omega_{n+1} \mid \cof(\alpha) = \omega_n \rB $.

\begin{lemma}

\label{lemmainduction}

Suppose $n \geq 1$, and both $R(n, G)$ and  $w \Diamond(S^n_{n+1})$ hold. Then $R(n+1, G) $ holds.
    
\end{lemma}

\begin{proof}

Let $P$ be as in the statement of $R(n+1, G)$. Using Lemma \ref{increasingunion} we can write $P = \bigcup_{\alpha < \omega_{n+1}} P_\alpha $ with further assumptions on the cofinality of the latter sets and their relation to the unbounded chain. Throughout the proof, we are going to use the quotient system discussed in section 3, so for any $f \in P$ we define $Q_f = \prod_{i \in \omega} G^{f(i)} / \bigoplus_{i \in \omega} G^{f(i)}$ and let $p^{f'}_f: Q_{f'} \to Q_f$ for the natural projection whenever $f \leq^\ast f'$. Finally, let 

\[     \Q = \lB Q_f ; p^{f'}_f ; P \rB.                  \]

Let $T = \lB s \in  2^{< \omega_{n+1}} \mid \textrm{supp}(s) \sbeq S^n_{n+1}   \rB$ with the ordering being end-extension. For $s \in T$ we let $r(s) = P_{\textrm{dom}(s)}  $.

We now define, for all $s \in T$, a sequence $\xx^s $ such that:

\begin{enumerate}
    \item $ \xx^s \in \Q \restriction r(s)^{(n)}     $ and is coherent.
    
    \item If $t \geq s$ then $\xx^t$ extends $\xx^s$.
    
    \item If $ \cof ( \textrm{dom}(s)) = \aleph_n  $ and $\uu \in \Q \restriction ( C \cap r(s))^{(n-1)} $ trivializes $\xx^s \restriction ( C \cap r(s) )^{(n)} $ then there exists $\epsilon \in \lB 0, 1 \rB $ such that there is no extension of $\uu$ to a trivialization of $ \xx^{s, \epsilon } \restriction (C  \cap P_{\textrm{dom}(s) +1}  )^{(n)} $.

\end{enumerate}

If $ \textrm{dom}(s) $ is limit, and $\xx^t $  has already  been constructed for all $t$ with $ \textrm{dom}(t) <  \textrm{dom}(s) $, we let $\xx^s = \bigcup_{\alpha < \textrm{dom}(s) } \xx^{s \restriction \alpha} $. 

Now for the successor step, since $\cof (r(s)) \leq \aleph_n$ then $\xx^s \in \Q \restriction r(s)^{(n)} $ is trivial by Goblot (Theorem \ref{flasquegoblot}). Then for $\epsilon \in \lB 0,1 \rB$, let $s, \epsilon = s^\frown \lB \epsilon \rB$ and $\ww^s \restriction r(s)^{(n-1)}  $ be a  trivialization of $\xx^s$ as in the induction, and $\vv^s  \in \Q \restriction (r(s, 0) )^{(n-1)}  $ be 

\[   \vv^s_{\Bar{\gamma} } = \begin{cases}

\ww^s_{\Bar{\gamma}} \,\,\,\,\,  \textrm{if} \,\,\, \Bar{\gamma}  \in  r(s)^{(n-1)}  \\ 

0   \,\,\,\,\,\,\,   \textrm{if} \,\,\, \Bar{\gamma} \in (r(s, 0) )^{(n-1)} \setminus r(s)^{(n-1)}.
    
\end{cases}         \]
Finally, if $ \cof (\textrm{dom}(s)) < \aleph_n $ we let $ \xx^{s, 0}  = \delta^n (  \vv^s  ) $.

Otherwise, $ \aleph_n = \cof (\textrm{dom}(s) ) \leq \cof(r(s))  \leq  \aleph_n $. In that case, we pick a family $\zz^s \in \Q  \restriction r(s)^{(n-1)}  $  that witnesses the induction hypothesis, that is, it is coherent on $P_{\textrm{dom}(s)} $ and nontrivial even when restricted to the $\omega_n$-chain unbounded in it that corresponds to a witness of the cofinality of $\textrm{dom}(s)$ (or equivalently to its $\leq^\ast$-closure) and let $\yy^s \in \Q \restriction (r(s, 0) )^{(n-1)} $

\[     \yy^s_{\Bar{\gamma} }  = \begin{cases}

\zz^s_{\Bar{\gamma}} \,\,\,\,\,  \textrm{if} \,\,\, \Bar{\gamma}  \in  r(s)^{(n-1)}  \\ 

0   \,\,\,\,\,\,\,   \textrm{if} \,\,\, \Bar{\gamma} \in (r(s, 0) )^{(n-1)} \setminus r(s)^{(n-1)}.

\end{cases}                              \]
Finally, we let:  $ \xx^{s, 0} = \delta^n (  \vv^s  )    $  and $\xx^{s, 1} = \delta^n (\vv^s + \yy^s )  $.

Let us now check that conditions $(1)$ to $ (3)$ are satisfied. An increasing union of coherent sequences is coherent, and it extends the sequences of which it is a union, so conditions $1$ and $2$ are preserved at limit steps. They are also preserved at successor steps because trivial sequences are coherent and 

\[   \delta^n (\vv^s + \yy^s) \restriction r(s)^{(n)} = \delta^n (\vv^s)   \restriction r(s)^{(n)} +    \delta^n (\yy^s) \restriction    r(s)^{(n)}     =     \delta^n (\vv^s)    \restriction r(s)^{(n)}    =    \]
\[  =   \delta^n (\vv^s    \restriction r(s)^{(n)}  )  = \delta^n ( \ww^s) = \xx^s.       \]

The only thing left is to prove condition $3$. Suppose by contradiction that $\uu^0 \restriction (C \cap r(s) )^{(n-1)} = \uu = \uu^1 \restriction (   C \cap r(s)    )^{(n-1)} $, $\delta^n (\uu^0) = x^{s, 0} \restriction ( C \cap r(s,0)) )^{(n)}  $ and $\delta^n (\uu^1) = x^{s, 1} \restriction (C \cap r(s,0) ) )^{(n)}   $.

Let $f = \max \lB C \cap r(s,0) \rB$ and $\uu  \in \Q \restriction P_{\textrm{dom}(s)+1}^{(n-1)}$. We define $d_f \in \Q \restriction (C \cap r(s))^{(n-2)}$ by $( d_f (\uu))_{\Bar{h}}  = \uu_{\Bar{h}, f }  $,  Then, if $\Bar{g}$ ranges over $ (C \cap r(s))^{(n-1)} $, we have (omitting restrictions for readability):

\[ \delta^{n-1} ( d_f ( \uu^0)  )_{\Bar{g}} = \sum_{0\leq i \leq n-1}  (d_f \uu^0)_{\Bar{g}^i}  = \sum_{0\leq i \leq n-1}  \uu^0_{\Bar{g}^i, f} = (-1)^{n-1} \uu^0_{\Bar{g}} + \xx^{s, 0}_{\Bar{g}, f} = (-1)^{n-1} \uu^0_{\Bar{g}} + (-1)^n \ww_{\Bar{g}}.   \]
By the same reasoning 

\[   \delta^{n-1} ( d_f ( \uu^1)  )_{\Bar{g}}  =  (-1)^{n-1} \uu^0_{\Bar{g}} + (-1)^n \cdot ( \ww_{\Bar{g}} + \zz^s_{\Bar{g}} ),    \]
so that, since $\uu^0 $ and $\uu^1$ agree on $(C \cap r(s))^{(n-1)} $

\[ \delta^{n-1} ( d_f ( \uu^1 - \uu^0 )  ) = (-1)^n \zz^s \restriction (C \cap r(s))^{(n-1)},  \]
contradicting the fact that the restriction of $\zz^s$ to the unbounded chain is nontrivial.

It should be noticed that the case $n=1$ is computationally identical but conceptually different insofar as we need to stipulate that a family indexed by $0$-tuples is a singleton and that a function defined on $I_f$ can be regarded as a trivialization of a $1$-family below $f$ as it can be extended to the entire grid $\omega \times \omega$. 

This concludes the construction of the $\xx^s$ for $s \in T$. For $g \in 2^{\omega_{n+1}} $ with $\textrm{supp}(g) \sbeq S^n_{n+1}$ let

\[   \xx^g = \bigcup_{\lambda < \omega_{n+1} } \xx^{g \restriction \lambda}.      \]
Note that $\xx^g \in \Q^{(n)} $ is coherent (increasing union of coherent sequences). 

Now we need a club $D \sbeq \omega_{n+1} $ such that for all $\lambda \in D$ and $s \in T$ with $\textrm{dom}(s) = \lambda $, the latter codes a sequence $\uu^s \in \Q \restriction (C \cap r(s))^{(n-1)} $ and moreover: 

\begin{itemize}

    \item If $\lambda \in D$ then every $\uu^s \in \Q \restriction (C \cap r(s))^{(n-1)}$  is coded by some $s: \lambda \to 2$.

    \item If $s \leq t$ then $\uu^t$ extends $\uu^s$.
    
\end{itemize}

For this, we will use 

\[ D = \lB  \lambda < \omega_{n+1} \mid \omega_n \cdot \lambda = \lambda   \rB  ,       \]
since for any two elements $\lambda, \lambda' \in D$ there is a subset of order-type at least $\omega_n$ between them in $\omega_{n+1}$. Now let $\lambda_+ $ be the successor of $\lambda $ in the increasing enumeration of the elements of $D$. Then there are enough end-extensions of $s : \lambda \to 2$ to an element in $2^{\lambda_+}$ to code all the different extensions of $\uu^s \in \Q \restriction (C \cap P_{\lambda})^{(n-1)} $ to a sequence $\uu^t \in \Q \restriction  (C \cap P_{\lambda_+})^{(n-1)} $. Indeed, these are at most $(\aleph_0^{\aleph_0})^{\aleph_n} = 2^{\aleph_n} $. The base case of the transfinite induction follows similarly from the fact that $  \omega_n \leq \lambda_0 = \min D$ Since we are dealing with a club, the limits of the enumeration are limits in $\omega_{n+1}$ of the previous elements, and we can continue the process at limit steps just by taking unions.
Finally, for $b : \omega_{n+1} \to 2$, we let

\[ \uu^b =  \bigcup_{\lambda \in D} \uu^{b \restriction \lambda}.                                     \]
Notice that this implies that for all $\uu \in \Q \restriction C^{(n-1)}$ there is a $b $ that codes it and for every such $b $ we have 

\[  \uu^{b \restriction \lambda}  = \uu^b  \restriction (C \cap P_\lambda)^{(n-1)}               \]
on a club.

Before defining our guessing function, we need the following claim. Such an observation and the consequent definition of $F$ seem needed in the argument of \cite{velickovic2021non} as well, yet they are not present.

\begin{claim}

Let $\uu \in \Q \restriction (C  \cap P_\beta)^{(n-1)} $ for some $\beta < \omega_{n+1}$. Then there is at most one $s \in 2^\beta$ such that $\uu$ trivializes $\xx^s \restriction (C \cap P_\beta)^{(n)}$.

\end{claim}

\begin{proof}[Proof of Claim]

Assume, towards a contradiction, that there exist distinct $t$ and $t'$ such that $\uu$ trivializes both $\xx^t \restriction (C \cap P_\beta)^{(n)}$ and  $\xx^{t'} \restriction (C \cap P_\beta)^{(n)}$. Let $\alpha \in S^n_{n+1} \cap \beta $ be least such that $t(\alpha) \neq t'(\alpha) $. Then $\uu \restriction (C \cap P_\alpha)^{(n-1)}  $ trivializes $x^t \restriction (C \cap P_\alpha)^{(n)} = x^{t \restriction \alpha} \restriction (C \cap P_\alpha)^{(n)}  = x^{t' \restriction \alpha} \restriction (C \cap P_\alpha)^{(n)} = x^{t'} \restriction (C \cap P_\alpha)^{(n)}$. 

For the same reasons,  $\uu \restriction (C \cap P_{\alpha+1})^{(n-1)}  $ trivializes both $x^t \restriction (C \cap P_{\alpha+1})^{(n)} = x^{t \restriction \alpha +1} \restriction (C \cap P_{\alpha+1})^{(n)} $ and $x^{t'} \restriction (C \cap P_{\alpha+1})^{(n)} = x^{t' \restriction \alpha+1} \restriction (C \cap P_{\alpha+1})^{(n)} $. This contradicts property 3 of our construction.
\end{proof}

Then let $T \restriction D$ be the set of all nodes whose level belongs to $D$. We define $F: T \restriction D \to 2$ as follows. If $\uu^s$ trivializes $\xx^{t} \restriction (C \cap P_{\textrm{dom}(t)})^{(n)} $ for some $t$ such that $\textrm{dom}(s) = \textrm{dom}(t)$, which is unique by the previous claim, then $F(s) = \epsilon \in \lB 0, 1 \rB$ such that $\uu^s$ does not extend to a trivialization of $x^{t, \epsilon}$. Otherwise, let $F(s)$  take an arbitrary value in $ \lB 0, 1 \rB $.  

Now $w \Diamond (S^n_{n+1}) $ means precisely that for all $ F: 2^{<\omega_{n+1}} \to 2$  there exists $ g: \omega_{n+1} \to 2 $ such that for all $ b: \omega_{n+1} \to 2  $.

\[  S =  \lB \lambda \in S^n_{n+1} \mid g(\lambda) = F(b \restriction \lambda) \rB                 \]
is stationary. Then $\xx^g \restriction C^{(n)} $ is nontrivial because for any putative trivialization  $ \uu^b $  we get $\lambda  \in S \cap D$ such that $g(\lambda) = F(b \restriction \lambda)$. Now given our contradiction assumption,

\[  \uu^{b \restriction \lambda}  =  \uu^b \restriction (C \cap 
 P_\lambda)^{(n-1)}   \]  
trivializes

\[ \xx^g \restriction (C \cap P_\lambda)^{(n)}  = \xx^{g \restriction \lambda}   \restriction (C \cap P_\lambda)^{(n)}.    \] 
So, if we use the definition of $F$ above for $s= b \restriction \lambda$ and $t = g\restriction \lambda $ we find that cannot be extended to a sequence in $\Q \restriction (C \cap P_{\lambda +1} )^{(n-1)} $ that trivializes $\xx^{g \restriction (\lambda + 1)} \restriction  (C \cap P_{\lambda +1} )^{(n)} = \xx^g  \restriction   (C \cap P_{\lambda +1} )^{(n)} $, while $\uu^b \restriction (C \cap P_{\lambda +1} )^{(n-1)} $ is an extension that should do precisely that if $\uu^b$ is to be a trivialization.
\end{proof}

Finally, we can prove the following.

\begin{theorem}

\label{conclusioninduction}

Suppose $\mathfrak{d} = \aleph_n $, there exists an unbounded $\omega_n$-chain $C$ and $w \lozenge (S^k_{k+1})$ holds for $k < n$ and $2^{\aleph_0} < 2^{\aleph_1}$. Then for every nontrivial group $G$ there exists a coherent family $\xx \in \A^{(n)}_G$ for $G$ such that $\xx \restriction C^{(n)}$ is nontrivial. In particular, $\lim^n \A_G \neq 0 $.

\end{theorem}

\begin{proof}

The cardinal arithmetic assumption yields $R(1, G)$ by Lemma \ref{basecase}. Then by the weak diamonds and Lemma \ref{lemmainduction} one can deduce that $R(n, G)$ holds. Finally, if $\mathfrak{d} = \aleph_n $ and there exists an unbounded $\omega_n$-chain, then ${}^\omega \omega $ is among the posets $R(n, G)$ is about. According to Proposition \ref{ncornontriv}, nonvanishing derived limits are characterized by the existence of a coherent nontrivial family, and a coherent family whose restriction to some indexing subset is nontrivial is itself nontrivial.
\end{proof}

\begin{remark}
The cardinal arithmetic assumption $2^{\aleph_0} < 2^{\aleph_1}$ is not required for $G = \Z$, i.e. for the pro-group $\A$, as it is only needed in the base case for finite groups, as Remark \ref{remarkzeta} shows.

\end{remark}

\section{Cofinal rectangles and weak diamonds}

We will now show that all the set-theoretic assumptions that we needed for the proof of the main theorem are relatively consistent with $\mathrm{ZFC}$.

First, we need to introduce a ``nonlinear" finite support iteration of Hechler's forcings. The following forcing notion was first introduced in \cite{hechler1974existence}.

Let $\mathbb{Q}$ be any well-founded, countably directed poset. Let moreover $\mathbb{Q}^+ = \mathbb{Q} \cup \lB \ast \rB$ be the poset that consists of $\mathbb{Q}$ with one more element $\ast$ on top of it (i.e. greater than all other elements).

For each $a \in \mathbb{Q}^+$, we define a forcing notion $\mathbb{P}_a  $. First, let $\mathbb{Q}/a = \lB c \in \mathbb{Q} \mid c <_{\mathbb{Q}} a \rB $.
Suppose now that we have already defined $  \mathbb{P}_b$ for all $b <_\mathbb{Q} a$.

\begin{enumerate}

\item $p \in \mathbb{P}_a$ if and only if

\begin{enumerate}

    \item $p$ is a finite function with $\mathrm{dom}(p) \sbeq \mathbb{Q}/a$.

    \item For all $b \in \mathrm{dom}(p) $, $p(b) = (t, \dot{g}) $ where $t \in {}^{< \omega} \omega $, and $ \dot{g} $ is a $\mathbb{P}_b$-name for an element of ${}^\omega \omega$.

\end{enumerate}

\item  For $ p,q \in \mathbb{P}_a$, we write $p \leq q $ if and only if

\begin{enumerate}
    \item $ \mathrm{dom}(q) \sbeq \mathrm{dom}(p) $.

    \item For all $b \in \mathrm{dom}(q)$, if  $p(b) = (s,\dot{f})$ and $q(b) = (t,\dot{g}) $ then

\begin{enumerate}
    \item $t = s \restriction \mathrm{dom}(t)  $

    \item $ p \restriction (\mathbb{Q} / b) \Vdash_{\mathbb{P}_b} \mathrm{dom}(t) \leq n < \mathrm{dom}(s) \Rightarrow s(n)  >   \dot{g} (n) $.

    \item $   p \restriction (\mathbb{Q} / b) \Vdash_{\mathbb{P}_b}  \forall n \,\, \dot{f} (n) \geq \dot{g} (n)   $.

\end{enumerate}
    
\end{enumerate}

\end{enumerate}

The following is the case $\lambda = \omega$ of \cite[Theorem 1]{cummings1995cardinal}.

\begin{theorem}

\label{theoremcummings}

Suppose $\mathbb{Q}$ is any countably directed well-founded poset. Let $\mathbb{H}_{\mathbb{Q}} = \mathbb{P}_\ast $ be the poset corresponding to the maximum $\ast$ of $\mathbb{Q}^+$ as in the definition above.

\begin{enumerate}

    \item $\mathbb{H}_\mathbb{Q} $ is ccc. 

    \item $V^{\mathbb{H}_\mathbb{Q}} \models \mathbb{Q}     $ embeds cofinally into $( {}^\omega \omega, <^\ast ) $.

    \item If $V \models $ `` the bounding number of $\mathbb{Q}$ is $\beta$" then $V^{\mathbb{H}_\mathbb{Q}} \models \bfk = \beta $.

    \item If $V \models $ `` the dominating number of $\mathbb{Q}$ is $\delta$" then $V^{\mathbb{H}_\mathbb{Q}} \models \bfk = \delta $.

\end{enumerate}
    
\end{theorem}

We now focus on the specific instance of Hechler's forcing that we are going to need for our consistency result.

\begin{proposition}

Let $k \leq n < \omega$. If $V \models \mathrm{GCH} $ then  $\mathbb{H}_{\omega_k \times \omega_n}$ has size $\aleph_n$ and $V^{\mathbb{H}_{\omega_k \times \omega_n}} \models \cfk = \aleph_n$.

\end{proposition}

\begin{proof}

Let $(\omega_k , \omega_n) = \ast $ be the maximum element of the associated poset. One proves by mutual transfinite induction on the well-founded order relation that for all $(\omega_1, \omega_1) \leq (\alpha, \beta) \leq (\omega_k \times  \omega_n)$ we have $ \vert \mathbb{P}_{(\alpha, \beta)} \vert = \max \lB \vert \alpha \vert, \vert \beta \vert \rB  $ and $V^{\mathbb{P}_{(\alpha, \beta)}} \models \cfk = \max \lB \vert \alpha \vert, \vert \beta \vert \rB $. One implication follows from the fact that the nonlinear iteration is finite support and the other follows by the technique of nice names \cite[Lemma IV 3.11]{kunen2013set}.
\end{proof}

This allows us to apply the following lemma.

\begin{lemma}[see \cite{velickovic2021non}]

\label{lemmavelickovic}

Let $\kappa$ be an uncountable regular cardinal, and let $\lambda$ and $\mu$ be cardinals such that $\lambda^{< \kappa} < \mu$. Suppose $\mathbb{P}$ is a $\kappa$-cc poset of size $\lambda$. Let $\mathbb{Q} = \mathrm{Fn}(\mu \times \kappa, 2, \kappa)$. Then, for every stationary subset $S$ of $\kappa$ in $V$, $w \Diamond(S)$ holds in the generic extension by $\mathbb{P} \times \mathbb{Q}$.

\end{lemma}

This in turn allows us to prove the existence of a forcing extension that satisfies all of our requirements.

\begin{theorem}

\label{concofrect}

Let $k, n \leq \omega$. Then it is relatively consistent with $\mathrm{ZFC}$ that there is a cofinal set in ${}^\omega \omega $ isomorphic to $\omega_k \times \omega_n$ with respect to $<^\ast$, that $2^{\aleph_0} < 2^{\aleph_1}$, and $w \Diamond (S^i_{i+1}) $ for all $1 \leq i < n$.
    
\end{theorem}

\begin{proof}

We start with a ground model $ V \models \textrm{GCH}$. For $1 \leq i \leq n$, let $\mathbb{C}_i = Fn(\omega_{n+i} \times \omega_i, 2, \omega_i)$
and $ \mathbb{C} = \mathbb{H}_{\omega_k \times \omega_n} \times \mathbb{C}_1 \times ... \times \mathbb{C}_n   $ be the product forcing notion. Let $G = G_0 \times ... \times G_n$ be a generic filter over $\mathbb{C}$. It is well known \cite[Lemma V.2.6]{kunen2013set} that $\mathbb{C}_1 \times ... \times \mathbb{C}_n$ preserves cofinalities; moreover, it is $\sigma$-closed and so does not add any reals. Hence, by the product lemma \cite[Theorem V.1.2]{kunen2013set} $V[G] = V[G_1 \times ... \times  G_n][G_0] $ and so we have indeed a cofinal subset of shape $\omega_k \times \omega_n$ in the extension, as well as preservation of cofinalities and $2^\aleph_i = 2^{\aleph_{n+i}}$ for $0 \leq i \leq n$ (again, by nice names) as $\mathbb{H}_{\omega_k \times \omega_n} $ is absolute between $V$ and $V[G_1 \times ... \times G_n]$. 

As for $w \Diamond (S^i_{i+1}) $, note that for all $i < n$ the notion $\mathbb{C}_{i+2} \times \mathbb{C}_n $ is $\omega_{i+2}$-closed. Hence if we let $W = V[G_{i+2} \times ...  \times G_n]  $ we have $2^{\aleph_k} = \aleph_{k+1}$ in $W$. Now we apply \ref{lemmavelickovic} over $W$ with $\mathbb{P} = \mathbb{H}_{\omega_k \times \omega_n} × \mathbb{C}_1 \times ... \times \mathbb{C}_i$, $\mathbb{Q} = \mathbb{C}_{i+1}$, $\kappa = \omega_{i+1}$,
$\lambda = \omega_{n+i}$, $\mu =\omega_{n+i+1}$ , and $S = S^i_{i+1}$, we get the conclusion again by the product lemma. 
\end{proof}

\begin{proof}[Proof of Main Theorem]

Theorem \ref{concofrect} together with \ref{conclusioninduction} easily implies the consistency of $\lim^n \A_G \neq 0$ with the relevant values of $\mathfrak{b} $ and $\mathfrak{d}$.  The case $G = \Z $ is precisely the Main Theorem.

\end{proof}

\section{Appendix: the Roos complex}

As mentioned in the introduction, in this appendix we give a categorical proof of Roos' method of computation for derived limits. This consists of the argument presented in \cite[pp.345-348]{neeman2014triangulated}, with the addition of a proof for Claim \ref{claimneeman} and an explanation of how some suitable functors $F^i_a$ act on morphisms. These are contravariant, while Neeman claims that the corresponding functors in his proof are covariant. 

We start by defining one of the so-called Grothendieck's AB conditions.

\begin{definition}

Let $\lambda$ be an infinite cardinal. An abelian category $\cA \, $ is said to be  $AB4^*(\lambda) $ if it has all products of fewer than $\lambda$-many object and the products of less than $\lambda$-many exact sequences in $\cA \, $ is exact.   

\end{definition}

To make sense of the definition, one can define a category $\textrm{ExSeq(}\cA\textrm{)}$ of exact sequences in $\cA$ and compute the limit of a functor from a discrete category to it.

\begin{definition}[Simplicial and cosimplicial objects]

Let $\Delta$ be the category where objects are finite ordinals and morphisms are weakly increasing maps between them, and put $[n] \defeq \lB 0, ... , n \rB$. Then a simplicial object in the category $\mathcal{C}$ is a functor $X: \Delta^{op} \to \mathcal{C}$. \par
The $i$-th face is $d_i = X(\delta^i) : X([n]) \to X([n-1])     $ where $\delta^i : [n-1] \to [n] $ is the unique injection in the category $\Delta$ whose range omits $i \in [n]$.  \par

The $i$-th degeneracy is $s_i = X(\sigma^i) : X([n]) \to X([n-1])     $ where $\sigma^i : [n+1] \to [n] $ is the unique surjection in the category $\Delta$ such that $i \in [n]$ has two elements in its preimage. \par

Similarly, a cosimplicial object in the category $\mathcal{C}$ is a functor $Y: \Delta \to \mathcal{C}$. \par
Cofaces and codegeneracies are the morphisms $Y(\delta^i) $ and $Y(\sigma^i)$ respectively, for $i \in [n]$.

\end{definition}

\begin{definition}

The nerve of the category $\J$ is defined as the simplicial object $\mathscr{N}_{\bullet}(\J)  $ where $\mathscr{N}_k (\J) $ (the $k$-simplices) is the set of composable sequences of morphisms of length $k$

\[      i_0 \to i_1 \to i_2 \to ... \to  i_k  .                             \]

The $j$-th face map is that which composes morphisms at $i_j$ and the $j$-th degeneracy map is the one that adds the identity at $i_j$. \par
    
\end{definition}

\begin{definition}
An additive functor $F: \cA \to \cB $ between abelian categories is called \textit{ effaceable } if for any object $A$ in $\cA \, $  we can find a monomorphism $u: A   \to M$  such that  $F(u) = 0$.  

\end{definition}

\begin{definition}[$\delta$-functor]

Given two abelian categories \cA \, and \cB, a (cohomological covariant) \textit{$\delta$-functor } is a sequence of additive functors $ S^n : \cA \to \cB $ for $n \geq 0 $ and of connecting morphisms $\delta^n_E : S^n(X'') \to S^{n+1}(X')$ for each short exact sequence  in \cA

\[ E \defeq ( \,\, 0 \rightarrow X' \xrightarrow[]{f'} X \xrightarrow[]{f} X''  \rightarrow 0       \,\,   )   \]
making the following sequence exact

\[  0 \rightarrow S^0 (X') \xrightarrow[]{S^0(f')}  S^0 (X) \xrightarrow[]{S^0(f)}    S^0 (X'')  \xrightarrow[]{\delta^0_E} S^1(X')  \rightarrow  ...   \]

\[  ... \rightarrow    S^n (X') \xrightarrow[]{S^n(f')}  S^n (X) \xrightarrow[]{S^n(f)}    S^n (X'')  \xrightarrow[]{\delta^n_E} S^{n+1}(X')           \]
and such that if $\textbf{h} \defeq (h', h, h'') : E \to F  $ is a morphism of exact sequences as below

\begin{center}
    \begin{tikzcd}
    
    E \, = \, ( \,\, 0 \arrow[r] & X'\arrow[d, "{h'}"] \arrow[r] &  X \arrow[d, "{h}"]  \arrow[r] & X'' \arrow[d, "{h''}"] \arrow[r] &  0 \,\, )  \\
    
    F \, = \, ( \,\, 0 \arrow[r] & Y' \arrow[r] &  Y \arrow[r] & Y'' \arrow[r] &  0 \,\, )  
    
    \end{tikzcd}
    
\end{center}
then, for every $n \geq 0$, the following diagram commutes:

\begin{center}
    
  \begin{tikzcd}
    S^n(X'') \arrow[d,"S^n(h'')"] \arrow[r,"\delta^n_E"] &
    S^{n+1}(X') \arrow[d,"S^{n+1}(h')"] \\
    S^n(Y'') \arrow[r,"\delta^n_F"] & S^{n+1}(Y').
  \end{tikzcd}  

\end{center}

If $\textbf{S} = (S^n, \theta^n_E)$ and $\textbf{T} = (T^n, \omega^n_E)$ are two $\delta$-functors with homomorphisms $\theta^n_E: S^n(X'') \to S^{n+1} (X')$ and $\omega^n_E: T^n(X'') \to T^{n+1} (X')$, a \textit{morphism of $\delta$-functors} is a sequence of natural transformations $\phi^n : S^n \to T^n $ such that, for every exact sequence $E$ as above, the following diagram commutes:

\begin{center}

    \begin{tikzcd}
    S^n(X'')  \arrow[r, "{\theta^n_E}"]   \arrow[d, "{ \phi^n_{X''}}"]             &      S^{n+1}(X')  \arrow[d, "{ \phi^{n+1}_{X'} }"]               \\
     T^n(X'')     \arrow[r, "{\omega^n_E}"]                     & T^{n+1}(X')
    \end{tikzcd}  

\end{center}

An \textit{isomorphism} of $\delta$-functors is a morphism such that $\phi^n$ is a natural isomorphism for all $n$.

\end{definition}

\begin{definition}[Universal $\delta$-functor]

A $\delta$-functor $\textbf{S} = (S^n, \theta^n_E)$ is said to be \textit{universal} if for any other connected sequence of functors
$\textbf{T}= (T^n, \omega^n_E)  $ and any natural transformation $ \phi : S^0 \to T^0 $ there exists a
unique morphism $(\phi^n)_{n\geq 0}$ from $\textbf{S}$ to $\textbf{T}$ such that $\phi^0 = \phi$.

\end{definition}

\begin{remark}

\label{isouniversal}

If two $\delta$-functors $\textbf{S}$ and $\textbf{T}$ are both universal and $\phi :S^0 \to T^0 $ is a natural isomorphism, then there exists a unique (iso)morphism $(\phi^n)$ from $\textbf{S}$ to   $\textbf{T}$ such that $\phi^0 = \phi$. This is because the compositions of the unique morphisms of $\delta$-functors arising from $\phi $ and $\phi^{-1}$ must be the identity on $\textbf{S}$ and $\textbf{T}$ as they are the identity on level zero.

\end{remark}

\begin{fact}

\label{unilim}

$(\lim^n, \theta^n_E)$  form a universal $\delta$-functor, for some connecting morphisms $\theta^n_E$.

\end{fact}

The following lemma was first proved in Grothendieck's celebrated \textit{Tohoku paper} \cite{grothendieck1957quelques}. 

\begin{lemma}

\label{lemmaeffeceable}

Every $\delta$-functor $T= (T^i)_{i \geq 0} $ with $T^i$ effaceable for all $i \geq 1$ is universal. 

\end{lemma}

\begin{theorem}

Let $\lambda$ be an infinite cardinal. Let $\cA $ be an abelian category satisfying $AB4^* (\lambda) $. Suppose moreover that $\J$ is a small category with set of all morphisms of cardinality $< \lambda$. Then the functor 

\[   \lim : \cA^{\J^{op}} \to \cA    \]
has right derived functors $\lim^n$ and these can be computed as the cohomology of a complex.

\end{theorem}

\begin{proof}

Here we will construct the general version of the Roos complex and prove that its cohomology is the same as the derived limits by showing that it gives rise to a universal $\delta$-functor. \par

Now, if the cardinality of the total set of morphisms of $\J$ is strictly less than $\lambda$, then the same holds for $\mathscr{N}_k (\J) $, and since $\cA$ is $AB4^* (\lambda) $, given a functor $F: \J^{op} \to \cA $ we can form a complex 

\[   N_0(F) \xrightarrow[]{\partial_0}      N_1(F) \xrightarrow[]{\partial_1}    N_2(F) \xrightarrow[]{\partial_2} ...            \]
where 

\[   N_k (F) =   \prod_{  \lB i_0 \to i_1 \to ... \to i_k \rB  \in  \mathscr{N}_k (\J) }    F(i_0)  .                     \]
The differential $\partial_k : N_k (F) \to N_{k+1} (F)   $ is given by the alternating sum

\[   \partial_k =  \sum_{j=0}^{k+1}  (-1)^j \partial^j_k         ,                      \]
where $  \partial^j_k  $ is the morphism induced by the $j$-th face map of the nerve meaning that for $j > 0 $ we define the $i_0 \to ... \to i_{j-1} \to i_j \to i_{j+1} \to ... \to i_k \to i_{k+1} $-th component as the projection of $N_k(F) $ on the factor of index $i_0 \to ... i_{j-1} \to i_{j+1} \to i_k \to i_{k+1}$ and for $j=0$  we post-compose $F(i_o \to i_1 )$ to the obvious projection. One could similarly define maps induced by the degeneracies, and these maps constitute respectively the cofaces and codegeneracies of a cosimplicial object. The differential is therefore the alternating sum of the cofaces.  \par

We write $T^n (F)$ for the $n$-th cohomology of the Roos complex 

\[      N_0(F) \xrightarrow[]{\partial_0}      N_1(F) \xrightarrow[]{\partial_1}    N_2(F) \xrightarrow[]{\partial_2} ...   \]
Suppose now that we have a short exact sequence in $\cA^{\J^{op}}   $, namely, a sequence of functors $\J^{op} \to \cA$

\[  0 \to F' \to F \to F''  \to 0  . \]  
By $AB4^*(\lambda) $, the sequence

\[   0 \to N_k(F') \to N_k(F)  \to N_k(F'')  \to 0           \]
is also exact for all $k$. So we have a short exact sequence of complexes and by Zig-Zag Lemma a long exact sequence in cohomology. It follows that $(T^n)_{n\geq 0}$   is a $\delta$-functor. It is easy to see that $T^n$ commutes with products. The fact that $\lim$ can be expressed as the equalizer of two morphisms between products and the definition of the differentials give us a natural isomorphism $T^0 \cong \lim^0$.  \par

We now show that every object $F$ in $ \cA^{\J^{op}}   $ can be embedded in an object $F'$ such that $T^n (F') = 0$, for all $n >0$. This will imply the Theorem by Lemma \ref{lemmaeffeceable}, Remark \ref{isouniversal}, and Fact \ref{unilim}.   \par

For every object $i $ of $\J$ and every object $a$ of $\cA$ we define a functor $F^i_a : \J^{op} \to \cA  $  by the rule:

\[   F^i_a (j) = \prod_{\J(i, j) } a      . \]  
That is, $F^i_a (j)  $ is  product of copies of $a$ indexed over all morphisms from $i$ to $j$. The functor is defined on morphisms as follows. A morphism $\beta : j_1 \to j_2$ gets mapped to the morphism $\prod_{\mathcal{J}(i, j_2  )}  a  \to  \prod_{\mathcal{J}(i, j_1  )}  a   $ such that the component corresponding to the factor $\alpha : i \to j_1$ is $id_a \circ \pi_{\beta \circ \alpha}$, where $id_a$ is the identity on $a$ and $\pi_\gamma$, for $\gamma: i \to j_2 $, is the projection of $\prod_{\mathcal{J}(i, j_2  )}$ on the corresponding factor.

\begin{claim}
\label{claimneeman}

The complex

\[      N_0(F^i_a) \xrightarrow[]{\partial_0}      N_1(F^i_a) \xrightarrow[]{\partial_1}    N_2(F^i_a) \xrightarrow[]{\partial_2} ...         \] 
is homotopy equivalent to

\[  a \to 0 \to 0 \to ...           \]

\end{claim}

\begin{proof}[Proof of Claim]

First, one defines a category $\mathcal{C}$ whose objects are copies of $a$, one for each morphism $ i \to \cdot $ and such that 

\[ \textrm{Hom} ( \underset{\alpha : i \to j_1}{a}, \underset{\gamma : i \to j_2}{a} ) = \lB \beta: j_1 \to j_2 \mid   \beta \circ \alpha = \gamma   \rB.                        \]
Then note that 

\[  N_k (F^i_a) = \prod_{i_0 \to ... \to i_k } F^i_a (i_0) = \prod_{i_0 \to ... \to i_k }  \prod_{i \to i_0} a = \prod_{i \to i_0 \to ... \to i_k} a.                          \]
So the cohomological complex above is simply what we get by applying an operation reminiscent of the $\textrm{Hom}( \_, G) $ functor to get cohomology with coefficients in a group $G$ (here $a$ takes the role analogous to that of the target in the contravariant $\textrm{Hom}$) to the chain complex associated to the nerve of $\mathcal{C}$. Now call $[a]$ the trivial category on the initial object $a_{id: i \to i}$, $\eta : \mathcal{C} \to [a]$ the natural projection  and $\iota : [a] \to \mathcal{C}$ the inclusion. Now, $\eta \circ \iota = Id_{[a]}$ and $\iota \circ \eta$ is connected to $Id_{\mathcal{C}}$ by the obvious natural transformation represented in the commutative square below:

\begin{center}

  \begin{tikzcd}
    
    a_{i \xrightarrow[]{id} i} \arrow{d}[swap]{i \xrightarrow{\rho} j_1} \arrow{r}{i \xrightarrow[]{id} i} &    a_{i \xrightarrow[]{id} i} \arrow{d}{i \xrightarrow{\tau} j_2} \\
    
    a_{i \xrightarrow{\rho} j_1}  \arrow{r}{j_1 \xrightarrow[]{\sigma} j_2} & a_{i \xrightarrow{\tau} j_2}

    \end{tikzcd}

\end{center}

Now, by \cite[Proposition 7]{lee1972homotopy}, the existence of such a natural transformation implies that the induced simplicial maps $M(Id_{\mathcal{C}})$ and $M ( \iota \circ \eta       ) = M(\iota) \circ M(\eta) $ are homotopic. Since the same holds trivially between $M(Id_{[a]})$ and $M(\eta) \circ M(\iota) $, we can deduce the desired homotopy equivalence between the associated homological complexes, and then between the cohomological complexes ``with coefficients in $a$".   \end{proof}

So $T^n (F^i_a)  =0$ for all $n > 0$, and since $T^n$ commutes with product, the same holds for any product of $F^i_a $'s. Now, any functor $G$ can be embedded in the product over all $i$ of $F^i_{G(i)} $, the embedding of a factor into the product been given by the composition of the isomorphism  $ A\xrightarrow[]{\sim} A \times 0$ and the monomorphism $ A \times 0 \xhookrightarrow{} A \times B   $, both defined for general abelian categories. This concludes the proof.  \end{proof}

\bibliographystyle{amsplain}
\bibliography{bibliografia} 

\end{document}